\documentclass[pra,showpacs,amsmath,amssymb,amsfonts,lengthcheck,longbibliography,onecolumn]{revtex4-2}

\usepackage[letterpaper,top=2cm,bottom=2cm,left=3cm,right=3cm,marginparwidth=1.75cm]{geometry}
\usepackage[dvipsnames]{xcolor}
\usepackage{tcolorbox}
\tcbuselibrary{breakable}
\usepackage{graphicx, color, graphpap}% Include figure files
\usepackage{enumitem}
\usepackage{amssymb}
\usepackage{amsthm}
\usepackage{dsfont}
\usepackage{mathtools}
\usepackage{multirow}
\usepackage[colorlinks=true,citecolor=magenta,linkcolor=blue]{hyperref}
\usepackage[T1]{fontenc}
\usepackage{thmtools,thm-restate}
\usepackage{verbatim}
\usepackage{mathtools}
\usepackage{titlesec}
\usepackage{amsmath}
\usepackage[normalem]{ulem}
\usepackage[caption=false]{subfig} % old package, but only one compatible with RevTex4.2
\usepackage{verbatim}
\usepackage{wrapfig}
\usepackage{pifont}
\colorlet{mygreen}{green!70!black}
\colorlet{myred}{red!85!black}
\newcommand{\cmark}{{\color{mygreen}\ding{52}}}%
\newcommand{\xmark}{{\color{myred}\ding{56}}}%

\usepackage{xcolor}

\newcommand{\E}{\mathbb{E}}

\newtheorem{theorem}{Theorem}
\newtheorem{theorem_reset}{Theorem}
\newtheorem*{rot}{Rule of Thumb}

\begin{document}
\title{Lattice Random Walk Discretisations of Stochastic Differential Equations}

\author{Samuel Duffield}
\email{sam@normalcomputing.ai}
\author{Maxwell Aifer}
\author{Denis Melanson}
\author{Zach Belateche}
\author{Patrick J. Coles}

\affiliation{Normal Computing Corporation, New York, New York, USA}

\begin{abstract}
  We introduce a lattice random walk discretisation scheme for stochastic differential equations (SDEs)
  that samples binary or ternary increments at each step, suppressing complex drift and diffusion computations
  to simple 1 or 2 bit random values. This approach is a significant departure from traditional
  floating point discretisations and offers several advantages; including
  compatibility with stochastic computing architectures that avoid floating-point arithmetic in
  place of directly manipulating the underlying probability distribution of a bitstream,
  elimination of Gaussian sampling requirements, robustness to quantisation errors, and
  handling of non-Lipschitz drifts. We prove weak convergence and demonstrate the advantages
  through experiments on various SDEs, including state-of-the-art diffusion models.
\end{abstract}

\maketitle

\section{Introduction}

% \begin{wrapfigure}{r}{0.23\textwidth}
%     \vspace{-0.8cm}
%     \centering
%     \includegraphics[width=0.23\textwidth]{Figures_lrw/multimodal_comparison.png}
%     \caption{\textbf{Vizualisation of exact SDE, Euler-Maruyama and LRW} (with equal stepsize).}
% \end{wrapfigure}

%\textbf{Introduce and motivate SDEs}\\
Stochastic differential equations (SDEs) are a powerful tool for modelling a wide range of
phenomena across physics, finance, biology, and machine learning. In recent years, SDEs
have become particularly crucial in modern machine learning, where they form the
mathematical foundation of diffusion models—the breakthrough technology behind 
state-of-the-art image generation systems \cite{song2020score, esser2024scaling}. This is in addition to
established fields such as molecular dynamics \cite{leimkuhler2013rational} and Bayesian
inference \cite{horowitz1991generalized, duffield2024scalable} where the simulation of complex SDEs is the
key workhorse.

%\textbf{Solving SDEs is difficult}\\
Despite their widespread use, simulating SDEs presents significant computational challenges~\cite{higham2001algorithmic, burrage2004numerical}. 
SDEs are inherently continuous-time objects, which immediately creates difficulties for 
numerical implementation since digital computers can only perform discrete operations. Unlike 
ordinary differential equations, SDEs are augmented with continuous-time random noise.
This stochastic nature means that each simulation path is different, requiring multiple realizations to estimate 
statistical properties. Moreover, the interplay between deterministic drift and stochastic noise 
diffusion terms creates complex dynamics that can be sensitive to discretisation choices, 
particularly in high-dimensional systems~\cite{sarkka2019applied}. Additionally, SDE simulation is an inherently
sequential process and therefore receives little benefit from the parallelization of
modern GPU hardware.

%\textbf{Failings of existing approaches and hardware limitations}\\
The most widely used numerical method, Euler–Maruyama, approximates the continuous-time 
SDE by discretizing both the deterministic drift and stochastic diffusion terms. 
While conceptually simple and easy to implement, Euler–Maruyama suffers from several 
limitations: it assumes infinite numerical precision, making it vulnerable to 
quantisation errors in hardware implementations; and it can fail catastrophically for 
non-Lipschitz drift functions, which are common in modern applications~\cite{higham2012convergence, iguchi2024skew}. 
It also is not perfectly suited to digital hardware due to the assumption of infinite precision
and requirement for Gaussian sampling.

\begin{figure}[b]
    \centering
    \includegraphics[width=\textwidth]{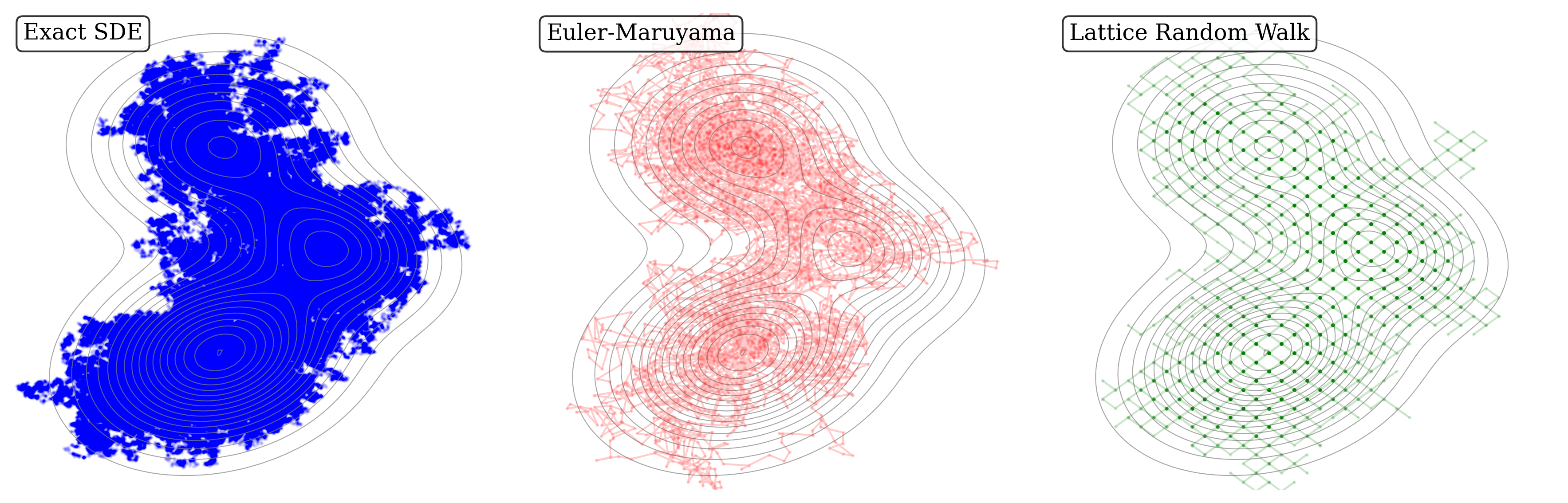}
    \caption{\textbf{Visualisation of exact SDE, Euler-Maruyama and LRW} (with equal stepsize).}
    \label{fig:toy}
\end{figure}

%\textbf{Contribution of the paper (LRW), and its benefits}\\
Specifically, our proposed lattice random walk (LRW) scheme for discretising SDEs, offers several advantages over existing methods (e.g., Euler–Maruyama), which include but may not be limited to:
\begin{itemize}
    \item \textbf{Enabling stochastic computing} for SDE simulation, unlocking potentially massive speedups on bespoke noise-based digital hardware.
    \item \textbf{No Gaussian sampling} required, which removes a non-trivial subroutine that involves transcendental functions and assumes infinite precision.
    \item \textbf{Robustness to quantisation}, the generated samples lie on an integer lattice and therefore their quantisation is directly incorporated into the weak error analysis as opposed to existing methods.
    Further the increment at each step is binary or ternary and therefore significantly more robust to error in the underlying drift and diffusion functions (such as that arising from quantisation).
    \item \textbf{Handling non-Lipschitz drifts}, as with the discretisation scheme of~\cite{iguchi2024skew}, but unlike Euler–Maruyama which is known to perform poorly for non-globally Lipschitz drifts \cite{higham2002strong}.
\end{itemize}
The LRW discretization as well as Euler-Maruyama and a continuous-time trajectory are visualized on a toy two-dimensional distribution in Fig.~\ref{fig:toy}, where one can clearly see the drastic difference in approaches between assuming infinite precision and traversing a lattice.

%\textbf{Structure of paper}\\
The paper is structured as follows: In Section~\ref{sec:lrw} we introduce the lattice
random walk discretisation, establish its (weak) convergence
and then describe related work in Section~\ref{sec:related_work}. In Section~\ref{sec:advantages} we expand
on each of the aforementioned advantages of LRW and discuss its implementation. In
Section~\ref{sec:experiments}, we present experiments demonstrating the advantages of LRW
as well as its scalability to state-of-the-art image diffusion models. In
Section~\ref{sec:discussion} we discuss the results, limitations as well as
future work and potential for impact.

\section{Lattice Random Walk}\label{sec:lrw}

%\textbf{Introduce SDE notation}\\
We consider methods for discretizing stochastic differential equations (SDEs), which in full generality have the form
\begin{equation}\label{eq:sde}
    dx = f(x, t) dt + \sigma(x, t) dw,
\end{equation}
where $f: \mathbb{R}^d \times [0, \infty) \to \mathbb{R}^d$ is denoted the \textit{drift}
vector, $\sigma: \mathbb{R}^d \times [0, \infty) \to \mathbb{R}^{d \times d}$ is
denoted the \textit{diffusion} matrix and $dw$ is a standard multivariate Brownian motion.
% with positive-definite $D(x,t) = 2 \; \sigma(x, t) \sigma(x, t)^T$.

%\textbf{Introduce Euler-Maruyama}\\
The most widely used numerical method for discretizing SDEs is the Euler–Maruyama (EM) scheme,
which approximates the continuous-time SDE \eqref{eq:sde} by discretizing both the
deterministic drift and stochastic diffusion terms. The scheme takes the form
\begin{equation}\label{eq:euler-maruyama}
    x_{t+\delta_t} = x_t + \delta_t f(x_t, t) + \sqrt{\delta_t} \sigma(x_t, t) \xi_t,
    \qquad \xi_t \sim \mathrm{N}(\xi \mid 0, I),
\end{equation}
for temporal stepsize $\delta_t \in (0, \infty)$.

%\textbf{Introduce LRW}\\
In stark contrast, our LRW scheme discretizes the SDE \eqref{eq:sde} by sampling a ternary-valued increment at each step.
Specifically, we consider the following multivariate discrete-time ternary-valued update:
% \begin{align}
%     x_{t+\delta_t} &= 
%     x_t + \Delta(x_t, t), \nonumber \\
%     \Delta_i(x, t) &=
%     \begin{cases}
%         -\delta_{x,i} \qquad & \eta_{t,i} < p_{-,i}, \\
%         0  &p_{-,i} \leq \eta_{t,i} < 1 - p_{+,i}, \\
%         +\delta_{x,i} &\eta_{t,i} \geq 1 - p_{+,i}.
%     \end{cases} \label{eq:lrw} \\
%     \eta_{t,i} &\sim U(0, 1), \nonumber
% \end{align}
\begin{align}
    x_{t+\delta_t} &= 
    x_t + \Delta(x_t, t), \nonumber \\
    \mathbb{P}[\Delta_i(x, t) = \Delta_i]
    &=
    \begin{cases}
        p_{-,i}(x,t), &\text{if } \Delta_i = -\delta_{x,i}, \\
        1 - p_{-,i}(x,t) - p_{+,i}(x,t), &\text{if } \Delta_i = 0, \\
        p_{+,i}(x,t), &\text{if } \Delta_i = \delta_{x,i}, \\
    \end{cases} \label{eq:lrw}
\end{align}
where $i$ indexes the coordinate of a $d$-dimensional vector. Here, the probability vectors depend on the SDE \eqref{eq:sde} and are defined as
\begin{equation}\label{eq:ternary-p}
    % p_\pm =
    p_\pm(x, t) = \frac{1}{2} \delta_t\delta_x^{-1} \left[\pm f(x, t) +\delta_x^{-1}\sigma(x, t)^2\right].
\end{equation}

% In an abuse of notation, we have assumed $\sigma(x,t)$ to be diagonal and behave as a vector with powers and multiplication understood elementwise.
Here, and from now on, we assume $\sigma(x,t)$ to be diagonal and behave as a vector with powers and multiplication understood elementwise, we elaborate on the restrictiveness of this assumption in Section~\ref{subsec:diag_diffusion}.
The parameter $\delta_x \in (0, \infty)^{d}$ is a spatial stepsize vector, whereas $\delta_t \in (0, \infty) $ is a temporal stepsize scalar (as in Euler-Maruyama).

%\textbf{Motivate LRW probabilities}\\
The motivation behind the probabilities \eqref{eq:ternary-p} becomes apparent upon the calculation of the first and second increment moments
\begin{align}\label{eq:ternary-moments}
\begin{split}
    \E[x_{t+\delta_t} - x_t \mid x_t] &= \delta_t f(x_t, t), \\
    \E[(x_{t+\delta_t} - x_t)^2 \mid x_t] &= \delta_t\sigma(x_t, t)^2.
\end{split}
\end{align}
Intuitively, we would then expect the discrete scheme to converge to the true SDE~\eqref{eq:sde} as $\delta_t \to 0$. This can be formalised with the notions of weak and strong convergence \cite{kloeden1992stochastic, pavliotis2014stochastic}.

\subsection{Weak Convergence}\label{subsec:weak_convergence}
Weak convergence measures how well a discretisation scheme approximates the statistical
properties of the true SDE solution. Unlike strong convergence, which concerns individual
sample paths, weak convergence
focuses on the accuracy of expectations of test functions evaluated at the discretized
solution. This in some sense is a necessary requirement for an SDE discretisation scheme
as it ensures expectations with respect to the true SDE are recovered as $\delta_t \to 0$.
Specifically, a method with weak order $p$ guarantees that expectations converge to the true SDE
solution at a rate of $O(\delta_t^p)$ as the temporal stepsize $\delta_t \to 0$, ensuring
that the discretisation faithfully recovers the underlying continuous-time stochastic process.

\begin{theorem}[Weak convergence of the LRW discretisation]
Consider the SDE \eqref{eq:sde} with drift function $f(x,t)$ and diagonal diffusion
matrix $\sigma(x,t)$ that are sufficiently smooth.
Let $\varphi: \mathbb{R}^d \to \mathbb{R}$ be a test function with bounded derivatives. 
Then the LRW discretisation~\eqref{eq:lrw} with spatial stepsize
$\delta_{x, i} = \Theta(\sqrt{\delta_t})$ has weak order 1, i.e.,
\[
\bigl|\E[\varphi(x_N)] - \E[\varphi(X(T))]\bigr| = O(\delta_t),
\]
where $x_N$ is the discretized solution at time $T = N\delta_t$ and $X(T)$ is the true SDE solution.
\end{theorem}
\begin{proof}
    Provided in Appendix~\ref{sec:weakproof}.
\end{proof}

Here we use the notation $\Theta(\cdot)$ to denote a function that is bounded above and below
by constant multiples of the argument.

The weak order 1 result ensures that the LRW discretisation
recovers the true SDE solution in the limit of small temporal stepsize and its rate
matches that of Euler-Maruyama. There are of course higher order methods that can
achieve higher weak order convergence rates, such the general class of stochastic Runge-Kutta methods \cite{debrabant2008classification}, see Table~\ref{tab:sde-methods}.
Often these are more complex to implement (such as requiring many queries to the drift
function per iteration) or require some restriction on the general SDE~\eqref{eq:sde}.

% \rb{one possible reorg would be to add a more explicit background (sub)-section presenting EM and the definition of Weak convergence as a Definition (+ why we care).  Then have a (one-line) Proposition that `EM is Weak order 1' in background (+ reference to higher order, but we focus on order 1.  Then LRW weak convergence could be a Proposition and it's clear that it matches EM. }

\subsection{Selecting $\delta_x$}\label{subsec:dx_constraints}

% \samd{ Add some discussion on whether $\delta_x$ needs to be constant or can
% depend on $t$ or even $x$}
% \samd{Relate to the $\delta_{x, i} = \Theta(\sqrt{\delta_t})$ constraint for weak convergence}

In the Euler-Maruyama discretisation \eqref{eq:euler-maruyama}, there is a single tuning parameter in $\delta_t$; for LRW we also have $\delta_t$ (which behaves in the same way as for EM controlling the level of temporal discretisation error and the number of iterations needed to reach a specified time $T$). However, in addition, we have the spatial stepsize $\delta_x$ which we now give some intuition on this new tuning parameter's behaviour and specification.

% The temporal stepsize $\delta_t$ controls the time discretisation error and how many iterations we need to get to a specified time $T$; the effect of this hyperparameter is transparent and well understood by the user as it mimics the stepsize in an Euler-Maruyama method. The spatial stepsize $\delta_x$ we have freedom to choose, and its effect is new and more nuanced, which we will discuss shortly.

For a valid ternary distribution, we have two constraints
\begin{equation}\label{eq:combined_constraints}
    \begin{aligned}
      \min(p_-,p_+) &\ge 0,\\
      p_- + p_+     &\le 1
    \end{aligned}
    \quad\Longrightarrow\quad
    \begin{aligned}
      \delta_x^{-1}\,\sigma(x,t)^2 &\ge |f(x,t)|,\\
      \delta_x^2                   &\ge \delta_t\,\sigma(x,t)^2
    \end{aligned}
  \end{equation}
where we have used \eqref{eq:ternary-p} and $|\cdot|$ represents elementwise absolute
value with the inequality required across all dimensions.

% Here we've used
% \begin{equation}\label{eq:sum_ps}
%     p_+ + p_- = \delta_t \delta_x^{-2}\sigma(x, t)^2.
% \end{equation}

% We can rewrite these constraints as an \textbf{allowable range} for $\delta_x$
% \begin{align}\label{eq:allowable}
%     \sqrt{\delta_t} \sigma(x,t) \leq \delta_x \leq \frac{\sigma(x,t)^2}{|f(x,t)|},
% \end{align}
% with a \textbf{feasibility condition} which ensures the range is nonempty
% \begin{align}\label{eq:feasibility}
%     \sigma(x,t)^2 \geq \delta_t |f(x,t)|^2.
% \end{align}
% This feasibility condition makes the stochastic nature of the discretisation apparent.
% Since we need $\delta_t > 0$, \eqref{eq:feasibility} implies we also need
% $\sigma(x,t)^2 > 0$, thus the discretisation is inherently stochastic and does not
% reduce to an ODE discretisation in the same way Euler-Maruyama does with $\sigma(x,t)=0$.

\begin{theorem}[Allowable range for $\delta_x$]
    For the probabilities in \eqref{eq:ternary-p} to represent a valid ternary distribution $\delta_x$ must satisfy
    \begin{align}\label{eq:allowable}
    \sqrt{\delta_t} \sigma(x,t) \leq \delta_x \leq \frac{\sigma(x,t)^2}{|f(x,t)|},
    \end{align}
     for specified $f$, $\sigma$ and $\delta_t$.
\end{theorem}
\begin{proof}
    Follows from basic manipulations of \eqref{eq:combined_constraints}.
\end{proof}

\begin{theorem}[Feasibility condition for $\delta_x$]
    For the range in \eqref{eq:allowable} to be non-empty we must have
    \begin{align}\label{eq:feasibility}
    \sigma(x,t)^2 \geq \delta_t |f(x,t)|^2.
    \end{align}
\end{theorem}
\begin{proof}
    Follows from setting equality in \eqref{eq:allowable}.
\end{proof}

This feasibility condition makes the stochastic nature of the discretisation apparent.
Since we need $\delta_t > 0$, \eqref{eq:feasibility} implies we also need
$\sigma(x,t)^2 > 0$, thus the discretisation is inherently stochastic and does not
reduce to an ordinary differential equation (ODE) discretisation in the same way
Euler-Maruyama does with $\sigma(x,t)=0$.
However, as we will see in Section~\ref{sec:lrw_odes}, we can adjust
the LRW discretisation into one that is consistent for deterministic ODEs.

% \subsubsection{Rule of thumb for selecting $\delta_x$}\label{subsubsec:dx_rule_of_thumb}

% The temporal stepsize $\delta_t$ controls the time discretisation error and how many iterations we need to get to a specified time $T$; the effect of this hyperparameter is transparent and well understood by the user as it mimics the stepsize in an Euler-Maruyama method. The spatial stepsize $\delta_x$ we have freedom to choose, and its effect is new and more nuanced, which we will discuss shortly.

% \paragraph{Satisfying the feasibility condition.}
The first criterion we have to satisfy is the feasibility condition \eqref{eq:feasibility}. The user will have freedom to choose $\delta_t$ and in some cases $\sigma(x,t)$. Thus, they can either decrease $\delta_t$ or increase $\sigma(x,t)$ to ensure the feasibility condition is met. In practice, we may not have intricate knowledge of $|f(x,t)|$ for each dimension and varying $x$ and $t$ thus the user chooses $\delta_t$ and perhaps $\sigma(x,t)$ to the best of their knowledge to ensure confidence in the feasibility condition is met. With the tradeoff being that decreasing $\delta_t$ increases the number of steps to get to a given time $T$ whilst increasing $\sigma(x,t)$ increases the noise in the SDE, which may be undesirable (if indeed it can be tuned).

% \paragraph{Satisfying the allowable range.}
Once $\delta_t$ and $\sigma(x,t)$ are set and we have confidence that the feasibility condition is met, we can turn to the selection of $\delta_x$. As mentioned, we typically have little knowledge on the range of $f(x,t)$ but more so on the range of $\sigma(x,t)$ (i.e. in the common case of fixed $\sigma(x,t) = \sigma$). We can often ensure the lower bound of \eqref{eq:allowable} is met since it only requires knowledge of $\sigma(x,t)$ and $\delta_t$, this gives us a rule of thumb for specifying $\delta_x$.
\begin{rot}Setting $\delta_x$ according to
\begin{equation}\label{eq:rot}
    \delta_x = \sqrt{\delta_t} \sigma_\text{max},
\end{equation}
for $\sigma_\text{max} \geq \sigma(x,t)$, ensures the lower bound in \eqref{eq:allowable} is satisfied. Then the upper bound is also satisfied so long as the feasibility condition \eqref{eq:feasibility} holds.
\end{rot}

\begin{theorem}[Reduction from ternary to binary]
    For a single LRW iteration, where the feasibility condition \eqref{eq:feasibility} holds, then setting
    \begin{equation*}
        \delta_x = \sqrt{\delta_t} \sigma(x,t),
    \end{equation*}
    reduces the ternary distribution \eqref{eq:ternary-p} to a binary distribution with $\mathbb{P}[\Delta_i(x, t) = 0] = 0$.
    Therefore if we have constant $\sigma(x,t)=\sigma$, the ternary update with rule of thumb reduces to binary at all iterations.
\end{theorem}
\begin{proof}
    Direct from \eqref{eq:combined_constraints}.
\end{proof}

% We note that in the case of fixed $\sigma(x,t)=\sigma$, then setting~\eqref{eq:rot} means the ternary update becomes a binary update (with $p_- + p_+ = 1)$.

We also note that, although we have described~\eqref{eq:rot} as a rule of thumb for setting
$\delta_x$ for a specified $\delta_t$, it also satisfies the condition
$\delta_x = \Theta(\sqrt{\delta_t})$ required for weak convergence in Section~\ref{subsec:weak_convergence}.

% \samd{Could we extend to 5-point or more? Why/why not?}

% \paragraph{When the constraints fail}

In practice, we likely cannot guarantee~\eqref{eq:feasibility} globally, particularly in
the case of non-globally Lipschitz drifts, Section~\ref{subsec:non-lipshitz}. In this
case we can still obtain a stable discretisation by clipping the probabilities
appropriately. That is, we can clip the coordinates of $\sigma(x, t)$ so that
$\delta_t \delta_x^{-2} \sigma(x, t)^2 \leq 1$ which ensures $p_- + p_+ \leq 1$.
Then clip the coordinates of $f(x,t)$ such that $-\delta_t \delta_x^{-2} \sigma(x, t)^2 \leq \delta_t \delta_x^{-1} f(x, t) \leq \delta_t \delta_x^{-2} \sigma(x, t)^2$
to ensure $p_-, p_+ \geq 0$. This clipping naturally modifies the moments~\eqref{eq:ternary-moments},
but as $\delta_t \to 0,$ the possibility of clipping disappears, therefore does not
affect weak convergence.

% \paragraph{Time-varying $\delta_x$}
For drift or diffusion functions that have a large
dynamic range over the course of a SDE trajectory, we can also generalize
$\delta_x \to \delta_x(t, x)$ to
be a function of $t$ or even $x$ so long as we still have $\delta_{x, i}(x,t) = \Theta(\sqrt{\delta_t})$. This is particularly useful for SDEs with time-varying
diffusion functions, such as variance-exploding diffusion models~\cite{song2020score}.

\subsection{Diagonal Diffusion Assumption}\label{subsec:diag_diffusion}

The LRW probabilities in \eqref{eq:ternary-p} as written make the
assumption that $\sigma(x,t)$ is diagonal. This is a strong assumption, but it covers
the majority of applications including Langevin-based sampling from Bayesian posteriors
\cite{duffield2024scalable} or molecular dynamics equilibrium distributions~\cite{leimkuhler2013rational},
as well as modern machine learning diffusion models~\cite{song2020score}.
We note that this diagonal diffusion limitation also features in
the discretisation scheme of Ref.~\cite{iguchi2024skew}.

Additionally, general SDEs can be converted into one with a constant and identity
diffusion matrix via a Lamperti transform. This is detailed in Appendix~\ref{sec:lamperti}
for the simpler case when the (dense) diffusion matrix depends on $t$ but not $x$, for
the complete case see~\cite{moller2010state}. Thus the majority of SDEs with dense diffusion
matrices (those that depend on $t$ but not $x$) can be handled easily by the LRW
discretisation via the Lamperti transform and therefore we leave the general case of a
space-dependent dense diffusion matrices as future work.

\begin{table}[h!]
\vspace{0.5cm}
  \centering
  \footnotesize
  \begin{tabular}{lcccc}
  \hline
  \textbf{Method} & \textbf{Weak order} &
  \begin{tabular}[c]{@{}c@{}}\textbf{Drift evals} \\ \textbf{per step}\end{tabular} &
  \begin{tabular}[c]{@{}c@{}}\textbf{Non-globally} \\ \textbf{Lipschitz drift}\end{tabular} & \textbf{Gaussian-free} \\
  \hline
  Euler--Maruyama \citep{kloeden1992stochastic} 
    & $1$ & $1$ & \xmark & \xmark \\
  
  Milstein \citep{Milstein1995,kloeden1992stochastic} 
    & $1$ & $1$ & \xmark & \xmark \\

  Two-point \citep{kloeden1992stochastic, Higham2001} & $1$ & $1$ & \xmark & \cmark \\
  
  % Heun (predictor--corrector) \citep{kloeden1992stochastic}
  %   &
  %   % \begin{tabular}[c]{@{}c@{}}$1$ (General) \\ $2$ (Constant diffusion)\end{tabular}
  %   $1^\dagger$
  %   & $2$ & \xmark & \xmark \\
  
  Stochastic Runge--Kutta \citep{BurrageBurrage1996, debrabant2008classification}
    & $2$ & $\geq 3$ & \xmark & \xmark \\
  
  Tamed Euler \citep{HutzenthalerJentzenKloeden2012}
    & $1$ & $1$ & \cmark & \xmark \\
  
  Implicit Euler-Maruyama \citep{hu1996semi}
    & $1$ & Implicit solve & \cmark & \xmark \\
  
  % Split-step Backward Euler \citep{HighamMaoStuart2002}
  %   & $1$ & Implicit solve & \cmark & \xmark \\
  
  \hline
  \textbf{Lattice random walk} \eqref{eq:lrw} & 1 & 1 & \cmark & \cmark \\
  \hline
  \end{tabular}
  \caption{\textbf{Common SDE discretisation methods}. Indicating weak convergence order, drift evaluations per step (calls to $f(x,t)$ to move from $t$ to $t+\delta_t$), suitability for non-Lipschitz drift, and whether Gaussian samples are required. Stochastic Runge-Kutta methods (which include the Heun method \cite{kloeden1992stochastic}) use multiple stages to achieve weak order 2~\cite{debrabant2008classification}. Implicit methods require a fixed-point solve that use a variable number of drift evaluations. Methods that require time-homogeneity (such as BAOAB-limit~\cite{leimkuhler2013rational} or the skew-symmetric method of \cite{iguchi2024skew}) are omitted.
  }
  \label{tab:sde-methods}
\end{table}

\subsection{Lattice Random Walk for Ordinary Differential Equations}\label{sec:lrw_odes}
An alternative LRW discretisation that is consistent for ODEs can be achieved with a variance that is quadratic in the stepsize $O(\delta_t^2)$ since then the variance decays 
faster than the mean as $\delta_t \to 0$.

Specifically if the increment's moments are modified to have the form
\begin{align}\label{eq:ternary-moments-ode}
\begin{split}
    \E[x_{t+\delta_t} - x_t \mid x_t] &= \delta_t f(x_t, t), \\
    \E[(x_{t+\delta_t} - x_t)^2 \mid x_t] &= \delta_t^2g(x_t, t)^2.
\end{split}
\end{align}
Then for some $g(x,t)$ independent of $\delta_t$ the discretisation will converge to the ODE $dx = f(x,t) dt$ as $\delta_t \to 0$.

Since ODEs are deterministic, we don't typically differentiate between strong and weak
error, since they coincide. In this case, we don't necessarily expect the above LRW
discretisation to have first order error $O(\delta_t)$ because of the required injected
noise. However, for many applications, such as Bayesian sampling and inference in
diffusion models in Section~\ref{subsec:exp_diffusion} a version of weak error
may still apply for ODEs where expectations are taken with respect to a random
initial distribution $p(x_0)$. A full error analysis in this setting of ODE
discretisation is left for future work.

\section{Related Work}\label{sec:related_work}

% \textbf{Connect to traditional lattice random walk work}\\
Lattice random walks have a rich history in physics and mathematics, dating back to early work on 
the development of discrete-time Markov processes on regular lattices~\cite{feller1950introduction}. These traditional approaches 
typically model discrete stochastic systems or natural phenomena that are inherently lattice-based, 
such as particle diffusion in crystalline structures~\cite{montroll1965random} or polymer dynamics on 
lattices~\cite{rubin1965random}.  In contrast,
our LRW discretisation is fundamentally different: rather than modeling discrete systems, 
we use lattice random walks as a computational tool to approximate continuous-time, continuous-space
SDEs that arise in diverse applications from molecular dynamics to machine learning.

% \MA{connect to early thermodynamic/math papers. See notion \url{https://www.notion.so/normal-ai/Lattice-OU-literature-review-14747d5923d1806c9675d015b93b7d20?source=copy_link}}

% \textbf{Euler-Maruyama}\\
SDEs represent an extremely broad class of continuous-time stochastic processes and
a large variety of numerical methods exist for simulating their evolution~\cite{sarkka2019applied}.
As discussed, the Euler-Maruyama method \eqref{eq:euler-maruyama} is the most widely used due to 
its simplicity and applicability to fully general SDEs~\eqref{eq:sde}. Alternative methods and their relation to LRW can be found in Table~\ref{tab:sde-methods}. Which notably includes stochastic Runge-Kutta methods \cite{debrabant2008classification} which use multiple steps to increase the weak convergence order as well as tamed~\citep{HutzenthalerJentzenKloeden2012} and implicit~\cite{hu1996semi} variants on Euler-Maruyama for improved stability to non-globally Lipschitz drift functions.

% A detailed review of
% existing SDE discretisation methods and their benefits and limitations is provided in
% Appendix~\ref{sec:existing_methods}.

% \textbf{Two-point stochastic discretisations}\\
There has also been developed stochastic discretisations
which modify the Euler-Maruyama method replacing the Gaussian noise source with a
binary-valued source (so-called two-point methods) or ternary-valued source (so-called
three-point methods) \cite{kloeden1992stochastic}. However here the discrete noise
source is only applied to the diffusion term and still requires a component of the form
$x_t + \delta_t f(x_t, t)$ for the drift term which contrasts significantly with the
LRW update~\eqref{eq:lrw} where a full iteration is binary or ternary
valued.

% \textbf{Skew-symmetric Barker}\\
The most closely related method to LRW is the skew-symmetric discretisation in~\cite{iguchi2024skew}, which builds on earlier work on spatially discretized stochastic processes \cite{bou2018continuous}.
This discretisation assumes a time-homogenous SDE $dx_t = f(x_t) dt + \sigma(x_t) dw_t$ and takes the form
\begin{equation*}
    x_{t+\delta_t} = x_t + b_t(x_t, \nu_t) \sqrt{\delta_t} \sigma(x_t) \nu_t,
\end{equation*}
where $\nu_t \sim \mathrm{N}(\nu \mid 0, I)$  and  $b_t(x_t, \nu_t)$
is a binary-valued random variable taking values in $\{-1, +1\}$ with probabilities
that ensure the weak order 1 convergence rate. The skew-symmetric discretisation
shares similarities with the LRW, notably that the drift computation is absorbed into the
sampling of a binary-valued random variable and that they share the same weak error and 
diagonal diffusion assumption. However, there remains notable differences including
that the skew-symmetric discretisation requires Gaussian sampling, doesn't produce
samples on a lattice, makes the assumption of time-homogeneity in the SDE, applies
the diffusion computation outside of the binary sampling as well as having a significantly
different (and more complex) form for the binary random variable probabilities.
These differences are key in enabling implementation on a stochastic computer
as well as application to time-inhomogeneous SDEs found in modern diffusion models.

\section{Advantages of Lattice Random Walk}\label{sec:advantages}

We now go into more detail on the advantages of the introduced LRW
discretisation over existing methods.

\subsection{Stochastic Computing}

%\rb{Could this be the last advantage and/or it's own section?  SC can be emphasized as motivation throughout Intro e.g., but this is a meatier section which may be useful to localize for interested readers} 

% \textbf{History, motivation}\\
% The concept of stochastic computing dates back to von Neumann's work on probabilistic logic~\cite{von1956probabilistic}
% in the early 1950s and was significantly advanced by Gaines~\cite{gaines1967stochastic} in the 1960s.
% The fundamental idea is that if we represent real numbers $p \in [0,1]$ as the underlying
% probability of a bit being 1, manipulations of the real numbers can be performed
% via simple bit-wise operations on a stream of random bits generated from $\text{Bernoulli}(p)$.

% \pj{Are we trying to sell LRW discretisation only for stochastic computing? Or do we want to say LRW is also useful for standard digital HW? We should probably clarify this.} \samd{The latter, agreed we should clarify this.}

% \pj{In this paragraph, we may want to remark that stochastic computing has been applied to other applications (citing literature) but still has not been unlocked for solving SDEs.}
% \pj{we may want to mention how stochastic computing has been applied to solving ODEs (with citations) but it has not yet been applied to SDEs, due to lack of theoretical methods (hence our approach fills this void)}
% \pj{We can use the phrase "co-design" to say how LRW is a discretisation scheme that we have co-designed to stochastic computing HW.}

% \textbf{Motivation for stochastic computing}\\
Stochastic computing~\cite{gaines1967stochastic, alaghi2013survey, gross2019stochastic}
represents an alternative paradigm to traditional deterministic digital
hardware. It is based on the idea that a real number can be encoded as the underlying
probability distribution of a random bitstream. This real number can then be manipulated
by performing operations on the bitstream and never having to store the real number itself
or apply floating point arithmetic. For this reason stochastic computing can be highly
time and energy efficient. However, it suffers from the issue that many applications
and pipelines require a real valued output and therefore the aggregation of the output
bitstream which can nullify the speedup of the internal operations. Our lattice random walk
discretisation utilizes complex underlying operations but only requires a binary or ternary
random variable at each step, therefore removing the aforementioned issue of aggregation
of the output and enabling stochastic computing for the key pipeline of SDE simulation.

% \textbf{Existing applications of stochastic computing}\\
Stochastic computing has been applied to a range of computational tasks, leveraging its
error resilience and low hardware cost. Notable applications include the numerical integration
of ordinary differential equations~\cite{liu2017hardware, camps2021stochastic},
neural network acceleration~\cite{alaghi2013survey,li2017towards}, and probabilistic
inference~\cite{gaines1967stochastic}. Yet to our knowledge stochastic computing has not
yet been applied to the simulation of general SDEs, critically due to the lack of
stochastic hardware-compatible theoretical methods which we unlock in this work.

% \textbf{Example, linear good}\\
As an example of the potential efficiencies of stochastic computing, consider a stochastic (unipolar) number $x$
(that is a random bit generated from $\text{Bernoulli}(x)$) and a stochastic number $y$
then a new stochastic number representing the product $xy$ can be computed with a simple
$\text{AND}$ gate $\text{AND}(x, y) = xy$.
Similarly we have $\text{NAND}(x, y) = 1 - xy$ and other basic gates, see Chapter 5 in~\cite{gross2019stochastic}
(although the $\text{NAND}$ gate is universal for boolean logic). This
makes linear operations (such as matrix multiplications) in stochastic computing
extremely efficient and can be applied in a single input bit to single output bit model
of computation. And importantly avoid any floating point arithmetic or quantisation error beyond that of the original stochastic numbers generation.

\begin{wrapfigure}[35]{r}{0.25\textwidth}
  \vspace{-0.3cm}
  \centering
  \includegraphics[width=0.23\textwidth]{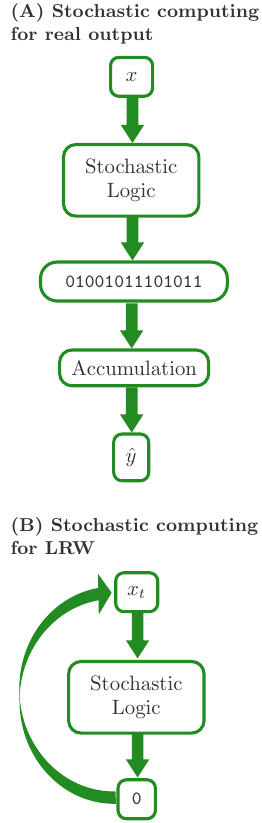}
  \caption{\textbf{Comparison of stochastic computing pipelines.} LRW enables stochastic computing for SDE simulation without bitstream accumulation.}
  \label{fig:stochastic_computing}
\end{wrapfigure}

% \textbf{Non-linear difficult}\\
However, non-linear operations are significantly more difficult and fundamentally require multiple
stochastic bits to generate a single output bit, that is multiple bits
$x_1, x_2, \dots \sim \text{Bernoulli}(x)$ are required
to generate a single output bit $y \sim \text{Bernoulli}(f(x))$ for non-linear $f(x)$. Still,
efficient and general-purpose non-linear univariate protocols have been developed including
finite-state machines~\cite{li2012synthesis}, Bernstein polynomials~\cite{qian2008synthesis}, Maclaurin expansions~\cite{parhi2016computing}, piecewise linear approximations~\cite{luong2019efficient} and recently Chebyshev polynomials~\cite{kind2025approximation}. The Kolmogorov-Arnold representation theorem
implies that any multivariate continuous function can be decomposed into the composition of matrix
multiplication and elementwise non-linearities, thus robust implementations of non-linear
univariate protocols combined with the linear arithmetic inherent to stochastic computing
is sufficient for universal (continuous) stochastic computation.

% \textbf{Why LRW unlocks stochastic computing}\\
A key limitation of stochastic computing is that it requires a conversion between the stochastic
domain (bitstreams) and the conventional domain (real numbers). This conversion typically requires
aggregating long bitstreams to estimate probabilities, which can nullify the computational
advantages gained during the stochastic operations. This bottleneck has limited the practical
application of stochastic computing. The setting of the LRW discretisation
fundamentally alleviates this limitation. Unlike traditional SDE
discretisation methods that require floating-point arithmetic for both drift and diffusion
computations, LRW reduces each iteration to sampling a simple discrete random variable, binary
or ternary (noting that a ternary random variable can be represented by two binary random
variables).  This output inherently avoids the issue of aggregating stochastic bitstreams
back to real numbers since the required output is binary or ternary, as depicted in Figure~\ref{fig:stochastic_computing}.

We note that a stochastic integrator, a fundamental circuit used in stochastic computing ODE solvers~\cite{gaines1967stochastic, liu2017hardware, liu2020dynamic} can be viewed as a ternary update. As such, the LRW reduction of SDEs to a simple ternary update serves as an analytical description of stochastic computing that fully incorporates its inherent stochasticity, thus enabling the use of randomness as a computational parameter instead of only a source of error as in stochastic computing implementations for deterministic processes~\cite{liu2017hardware, camps2021stochastic}.

In Appendix~\ref{sec:stochastic_ou}, we give a concrete stochastic computing protocol
for LRW on linear SDEs, i.e.\ Ornstein-Uhlenbeck (OU) processes $dx = -(Ax - b)dt + \sigma dw$. It utilises
a variant on the commonly used stochastic computing multiplexer (MUX)
primitive for scaled addition~\cite{gross2019stochastic}. By computing alias tables~\cite{walker1977efficient}
in a fast preprocessing step, categorical sampling can be performed in $O(1)$ time,
enabling a single iteration of LRW with work $O(d)$ and parallel time $O(1)$ versus 
traditional matrix-vector product arithmetic which has work $O(d^2)$ and parallel time 
$O(\log d)$~\cite{blelloch1996programming}. The overall complexities required to achieve weak
error $\epsilon$ are detailed in Table~\ref{table:mux}, which in both cases requires
$O(\epsilon^{-1})$ steps since both EM and LRW have weak order 1.

We note several caveats to this protocol: (i) the MUX-based approach imposes a maximal 
stepsize constraint (see Appendix~\ref{sec:stochastic_ou}), (ii) the protocol is specific to 
linear OU processes and does not directly extend to non-linear drifts, and (iii) OU processes can 
alternatively be sampled exactly with $O(d^3)$ preprocessing via matrix 
decomposition~\cite{duffield2024thermox}, though this requires significantly more complex 
floating-point operations whilst higher-order discretisations (see Table~\ref{tab:sde-methods}) can also improve the dependency on $\epsilon$.

The key advantage of the stochastic computing approach is that 
it replaces floating-point matrix-vector products with simple bit-level operations, 
potentially enabling significant energy and latency improvements on specialized hardware.
Notably CPUs and GPUs are designed for dense floating point arithmetic whereas the
stochastic approach moves the bottleneck to repeated categorical sampling via alias tables
which represents a fundamentally different pipeline ripe for specialized hardware acceleration.

\begin{table}[h]
\centering
\begin{tabular}{lccc}
  \hline
               & \textbf{Work}        & \textbf{\begin{tabular}[c]{@{}c@{}}Parallel\\ Span Time\end{tabular}}      & \textbf{\begin{tabular}[c]{@{}c@{}}Required\\ Threads\end{tabular}}  \\ \hline
Floating Point EM & $O(\epsilon^{-1}d^2)$     & $O(\epsilon^{-1}\log d)$ & $O(d^2)$       \\
Stochastic LRW & $O(d^2 + \epsilon^{-1}d)$ & $O(d + \epsilon^{-1})$   & $O(d)$         \\
\hline
\end{tabular}
\caption{\textbf{Complexity of sampling from an OU process with traditional vs stochastic computing protocols.}
We differentiate between \emph{work} (the total number of operations), \emph{parallel span time}
(the time taken to perform the operations assuming sufficient parallelism) and
\emph{required threads} (the number of threads required to achieve the stated parallel span).
The table shows the complexities required to achieve weak error $\epsilon$ on a dense OU
process with $d$ dimensions, including both preprocessing and iterative computation.}
\label{table:mux}
\end{table}

\subsection{No Gaussian Sampling}

Almost all existing SDE discretisation methods require Gaussian sampling with
the exception of the two-point (or three-point) methods from~\cite{kloeden1992stochastic}.
Gaussian sampling on digital hardware represents a non-trivial operation since it does
not have a native representation as randomness defined over bits. Typically, Gaussian
samples are generated using a Ziggurat method~\cite{harris2020array}, which is a fairly
involved (yet efficient) rejection sampling algorithm. In contrast, the LRW discretisation
requires only sampling a binary or ternary random variable, which is a much simpler operation,
which ties the discretisation implementation much closer to digital hardware.

\subsection{Robustness to (Quantisation) Error}

% \MA{Link to distillation/quantisation of SDEs for diffusion}

The form of the LRW discretisation suggests superior robustness to general error 
in the computation of the drift $f(x)$ and diffusion $\sigma(x)$ functions
compared to traditional methods like Euler-Maruyama. This robustness stems from the fundamental
difference in how errors propagate through the computation.

In Euler-Maruyama, errors in the drift and diffusion computations directly accumulate in the 
continuous-valued update~\eqref{eq:euler-maruyama}. 
Any quantisation or numerical error in computing $f(x_t, t)$ or $\sigma(x_t, t)$ compound 
over multiple iterations.

In contrast, the LRW discretisation contracts the entire drift and diffusion 
computation into the sampling of a binary or ternary random variable. This contraction provides 
a natural form of error suppression: errors in the underlying drift and diffusion functions only 
affect the probabilities $p_\pm$ in~\eqref{eq:ternary-p}, and these probabilities are then used 
to sample discrete outcomes. The discrete nature of the output means that small perturbations 
to the probabilities often result in the same discrete outcome, providing inherent robustness.

Quantisation error is a particularly notable example of this robustness: when the drift and 
diffusion functions are computed with limited precision (as is dictated by digital hardware), 
a portion of the resulting quantisation errors in $f(x,t)$ and $\sigma(x,t)$ are absorbed into the probability 
computation rather than being directly propagated as continuous-valued errors. This makes the 
LRW discretisation particularly well-suited for implementations on quantised 
hardware or in scenarios where computational precision is limited, as the discrete output 
structure naturally mitigates the impact of quantisation artifacts that would otherwise 
accumulate in continuous-valued methods which assume infinite precision.

Quantisation of the SDE and in particular the drift function is a very prominent consideration
in the field of large scale diffusion models where the drift function comprises a very large
neural network. Thus quantisation allows significant savings in time of execution,
energy consumption and financial cost \cite{li2023q} (this is independent of potential speedups
from stochastic computing where these advantages could be amplified significantly further).

\subsection{Non-Lipschitz Drifts}\label{subsec:non-lipshitz}

% \pj{Maybe add a sentence or two about the importance of non-Lipschitz drifts for certain applications, to give some motivation.}

We now turn to the concept of stability of SDE discretisation for a given non-zero stepsize
$\delta_t$. In this case, the discretised SDE does not have the same solution as the continuous-time SDE,
and the degree of separation depends on the discretisation method and the behaviour of
the drift and diffusion functions. Of particular interest are drift functions that are not globally Lipschitz continuous,
which informally are drift functions that cannot be bounded linearly in all areas of the state space.
In reality, global Lipschitz continuity is a very strong constraint \cite{hutzenthaler2015numerical, mao2013strong} that is not typically satisfied
by practically relevant SDEs such as those arising in chemistry~\cite{gillespie2000chemical}, finance~\cite{higham2012convergence},
modern machine learning diffusion models~\cite{song2020score} as well as Bayesian inference
as we will see in Section~\ref{subsec:exp_non_lipschitz}.

Formally, a drift function $f(x,t)$ is said to be \textit{globally Lipschitz continuous} if there exists a constant $L > 0$ such that
\begin{equation}
    \|f(x,t) - f(y,t)\| \leq L\|x - y\|,
\end{equation}
for all $x, y \in \mathbb{R}^d$ and all $t \geq 0$. 

Traditional proofs of the weak convergence of Euler-Maruyama require the drift function to be globally Lipschitz.
Whilst this can be relaxed, see~\cite{higham2002strong}, in practice Euler-Maruyama is known to perform poorly
for non-globally Lipschitz drifts. The fundamental issue 
is that the continuous-valued update~\eqref{eq:euler-maruyama} can lead to explosive behaviour when the drift 
function grows faster than linearly.

In contrast, the LRW discretisation exhibits superior stability for drifts that are not globally Lipschitz continuous. 
The key insight is that the discrete nature of the LRW update~\eqref{eq:lrw} naturally constrains the magnitude 
of each step to $\pm\delta_x$ or $0$, regardless of how large the drift function becomes. This bounded step 
size prevents the explosive behaviour that can occur in Euler-Maruyama.

More formally, we can consider the intuition of the skew-symmetric discretisation from~\cite{iguchi2024skew}
by considering the moments of the increment. For both Euler-Maruyama and LRW we have first moment
\begin{align*}
    \E[x_{t+\delta_t} - x_t \mid x_t] &= \delta_t f(x_t, t),
\end{align*}
but for the (diagonal) second moment we have
\begin{align}\label{eq:second_moment}
  \begin{split}
    \text{Euler-Maruyama: } 
    \quad \E[(x_{t+\delta_t} - x_t)^2 \mid x_t] &= \delta_t\sigma(x_t, t)^2  + \delta_t^2 f(x_t, t)^2, \\
    \text{Lattice Random Walk: } \quad \E[(x_{t+\delta_t} - x_t)^2 \mid x_t] &= \delta_t\sigma(x_t, t)^2.
  \end{split}
\end{align}

Thus for fixed stepsize $\delta_t$ the second moment of the Euler-Maruyama update
is unbounded for drifts lacking global Lipschitz continuity, whereas the second moment of the LRW update
is independent of the drift function.

% \samd{Discuss implicit Euler methods}
% \samd{Maybe we can support non-Lipschitz diffusion too}

\section{Experiments}\label{sec:experiments}

In this section, we will present some experiments demonstrating the advantages of the lattice random walk
discretisation. Specifically, we demonstrate robustness to quantisation error, stability for non-globally Lipschitz drifts
and scalability to large scale state-of-the-art diffusion models.

\subsection{Ornstein-Uhlenbeck Process}\label{subsec:exp_ou}

We compare the effect of floating point quantisation for Euler-Maruyama and the LRW discretisation
on traditional hardware with floating point arithmetic. Although as described in Appendix~\ref{subsec:reparam}
the LRW state can be reparameterised to integers, however the underlying arithmetic in the
drift and diffusion functions still typically requires floating point arithmetic, whose
quantisation error we will examine in this experiment.

We consider the multivariate Ornstein-Uhlenbeck process~\cite{aifer2024_TLA}
\begin{equation}\label{eq:ou}
    dx = -(A x -b) dt + \sqrt{2\mathcal{T}} dw,
\end{equation}
where $A \in \mathbb{R}^{d \times d}$ is a symmetric positive-definite matrix,
$b \in \mathbb{R}^d$ and $\mathcal{T}>0$ is a scalar temperature. The OU process has stationary
distribution $\mathrm{N}(x \mid A^{-1}b, \mathcal{T} A^{-1})$.

\begin{figure}[htbp]
  \centering
  \begin{minipage}[c]{0.48\textwidth}
    \centering
    \includegraphics[width=\textwidth]{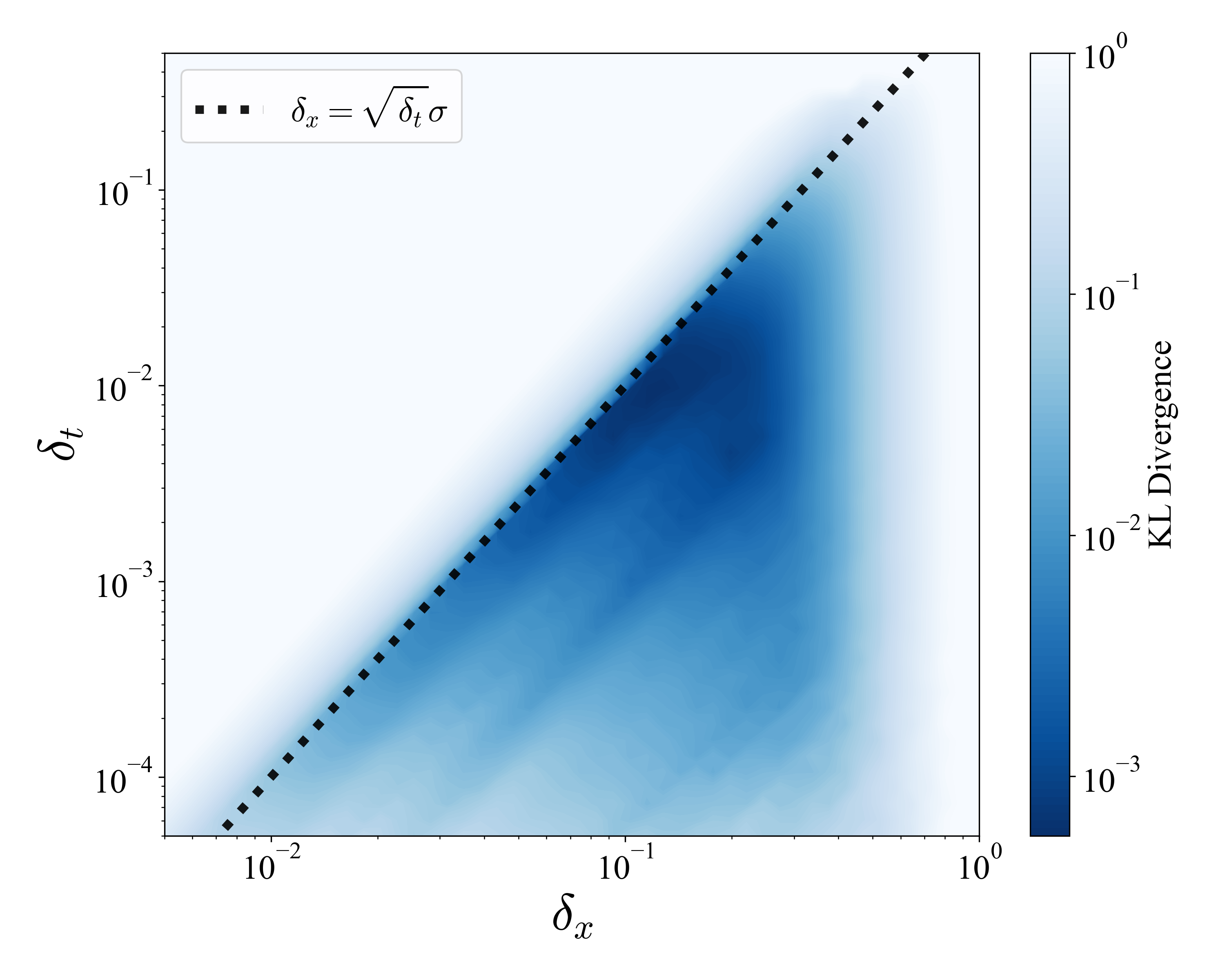}
    \caption{\textbf{Sensitivity to $\delta_x$ for an OU process}. KL divergence
    between true stationary and LRW distributions, as a function of $\delta_t$ and $\delta_x$. Averaged over 10 different seeds. The dotted line corresponds to the binary update condition~\eqref{eq:rot}.}\label{fig:delta_x}
  \end{minipage}
  \hfill
  \begin{minipage}[c]{0.48\textwidth}
    \centering
    \includegraphics[width=\textwidth]{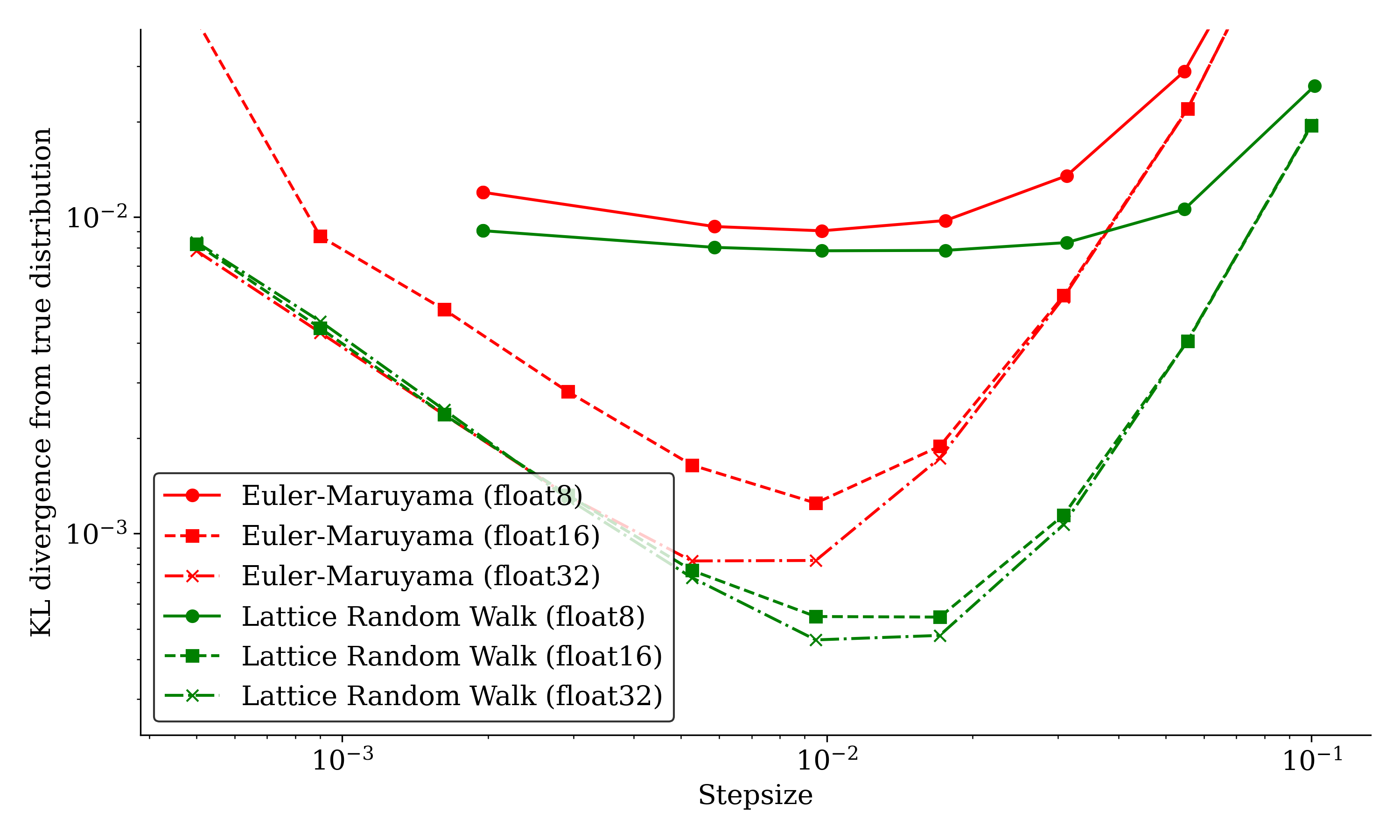}
    \caption{\textbf{Robustness to quantisation for an OU process}. KL divergence
    between true stationary distribution and the discretisations. Averaged over 50 different seeds
    and randomly sampled OU parameters.}\label{fig:quantisation}
  \end{minipage}
\end{figure}

We set $d=3$, $\mathcal{T}=0.5$ and sample $A = Z Z^\text{T} + \mathbb{I}$ where $Z\in \mathbb{R}^{d \times d}$ with
$Z_{ij} \sim \mathrm{N}(Z \mid 0, 1)$ and $b \sim \mathrm{N}(b \mid0, \mathbb{I})$. For each seed and discretisation scheme
we run $10^6$ steps (noting that it is the number of steps that is fixed rather than the time duration $T$) starting from $x_0 = 0$ discarding the first third 
as a burn-in. We measure accuracy with $\text{KL}[\hat{\pi}, \pi]$ where $\hat{\pi}$ is the Gaussian distribution
constructed from the empirical mean and covariance of the samples and $\pi$ is the true stationary distribution.

We start by examining the rule of thumb from Section~\ref{subsec:dx_constraints} in Figure~\ref{fig:delta_x}.
Here we repeat the experiment over a range of stepsizes $\delta_t$ and $\delta_x$. We can see that
setting $\delta_x$ at $\sqrt{\delta_t} \sigma$ or slightly larger results in the best performance
justifying the rule of thumb (noting that $\sigma$ is constant in this experiment). Whilst when the constraints are not met (i.e. above the rule of thumb line), performance degrades rapidly.

We then adopt this rule of thumb and set $\delta_x = \sqrt{2\delta_t \mathcal{T}}$ resulting in a binary update
for the quantisation experiment in Figure~\ref{fig:quantisation} where we vary the floating point precision
of the underlying drift and diffusion arithmetic using 8, 16 and 32 bits.
We observe that LRW is as accurate and generally more so than Euler-Maruyama for all stepsizes and precision.
Particularly for large stepsizes $\delta_t$ we observe that the LRW discretisation performs significantly better than the Euler-Maruyama
discretisation. For small stepsizes and high precision (32 bits) the Euler-Maruyama discretisation
can match the performance of the LRW discretisation but not for low precision (16 bits). For sufficiently
small stepsize the error from lack of exploration (small evolution time) dominates.
Overall, we see negligible degradation in performance for the LRW discretisation
from 32 to 16 bit precision.

It's also worth commenting that the LRW discretisation (with rule of thumb $\delta_x$) is significantly more robust to the choice of the temporal stepsize $\delta_t$, as demonstrated by the flatter bottom of the curves in Figure~\ref{fig:quantisation}.

\subsection{Poisson Random Effects Model (Non-globally Lipshitz)}\label{subsec:exp_non_lipschitz}

We now investigate sampling an SDE with a non-globally Lipschitz drift function. Following~\cite{iguchi2024skew}, we consider overdamped Langevin sampling from
a Bayesian Poisson random effects model. In particular, we consider the SDE 
\begin{align*}
  dx &= - \nabla U(x) dt + \sqrt{2} dw, \\
  U(x) &= J \sum_{i=2}^{d+1} e^{x_i} - \sum_{i=2}^{d+1} \sum_{j=1}^J y_{ij} x_i
  + \frac12 \sum_{i=2}^{d+1} (x_i - x_1)^2 + \frac{x_1^2}{2\sigma_1^2},
\end{align*}
where $x_1 \sim \mathrm{N}(x \mid 0, \sigma_1^2)$, $x_i \sim \mathrm{N}(x \mid x_1, 1)$,
$y_{ij} \sim \text{Poisson}(e^{x_i})$, for $i=2,\ldots,d+1$ and $j=1,\ldots,J$.

\begin{wrapfigure}[16]{r}{0.4\textwidth}
\vspace{-0.2cm}
  \centering
  \includegraphics[width=0.4\textwidth]{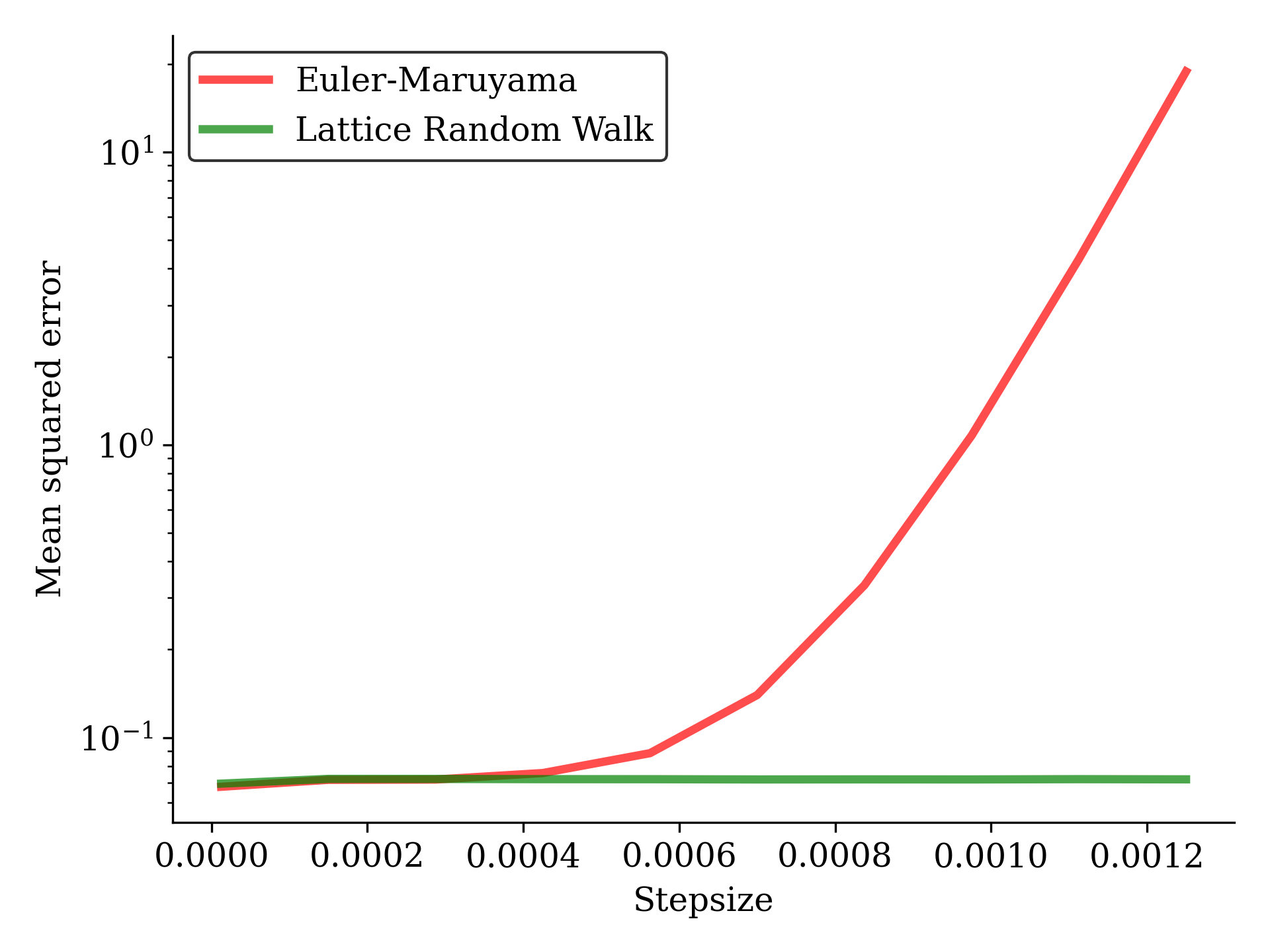}
  \caption{\textbf{Robustness to stepsize for a non-Lipschitz Poisson random effects model}.
  Error from the ergodic mean of the generated samples and the true parameter, averaged over 50 different seeds.}
  \label{fig:poisson}
\end{wrapfigure}

% We highlight that the drift function is not globally Lipschitz continuous for this model.

Following~\cite{iguchi2024skew} we set $d=51$ but only consider $x_0$ as an
interest parameter, we set $J=5$ and $\sigma_0 = 10$. We set underlying true parameters
$x^*_0 = 5$ and sample the rest $x^*_i \sim \mathrm{N}(x \mid x^*_0, 1)$. We initiate
the discretisations at the true parameters (therefore do not apply a burn-in) and
run for 50k steps. The posterior distribution in $x_0$ is well concentrated around the true parameter
$x_0^*$ and thus we report the mean squared error from the ergodic mean of the generated samples
and the true parameter. For each stepsize, we repeat each trajectory generation 100 times with
different seeds. As above we set $\delta_x = \sqrt{2\delta_t}$ using the rule of thumb \eqref{eq:rot} justified in Figure~\ref{fig:delta_x}.

We can see from Figure~\ref{fig:poisson} that the LRW discretisation
performs well across all stepsizes whereas the error in the Euler-Maruyama discretisation
explodes for larger stepsizes, as also observed in prior works \cite{hutzenthaler2011strong, iguchi2024skew}.
% \lh{It might be good to briefly cite a few papers here that have observed the same issues with EM experimentally. This would give a little more body to the experiments section}

\subsection{Diffusion Model}\label{subsec:exp_diffusion}

In this experiment we deploy lattice random walk discretisation to a large-scale
state-of-the-art image generation model. We use the popular Stable Diffusion 3.5 model~\cite{esser2024scaling}
which has over 8 billion parameters. Stable Diffusion 3.5 is technically a flow-matching model~\cite{karras2022elucidating} and differs in
training than that of a continuous-time diffusion model~\cite{song2020score}. This class of models are
usually viewed as a deterministic ordinary differential equation at inference time.
However, as described in Appendix~\ref{sec:from_flows_to_sdes} (which expands on the exposition from \cite{karras2022elucidating}), flow-matching models can be recast as SDEs of the form
\begin{equation}\label{eq:diffusion_sde}
    dx = \dot{\varsigma}(t)\varsigma(t) s(x, t) dt + \alpha(t) s(x,t) dt + \sqrt{2 \alpha(t)}dw,
\end{equation}
where $s(x,t)$ is the learnt score function, $\varsigma(t)$ is a noise schedule (also determined by the model training),
$\dot{\varsigma}(t)$ is its time derivative and $\alpha(t)$ is a tuning parameter that controls
the level of noise added during the trajectory. Remarkably the marginal distributions
$p(x,t)$ are the same for all choices of $\alpha(t) \geq 0$~\cite{karras2022elucidating}.

\begin{figure}[h]
  \centering
  \subfloat[$\delta_t=1/50$.]{\includegraphics[width=\linewidth]{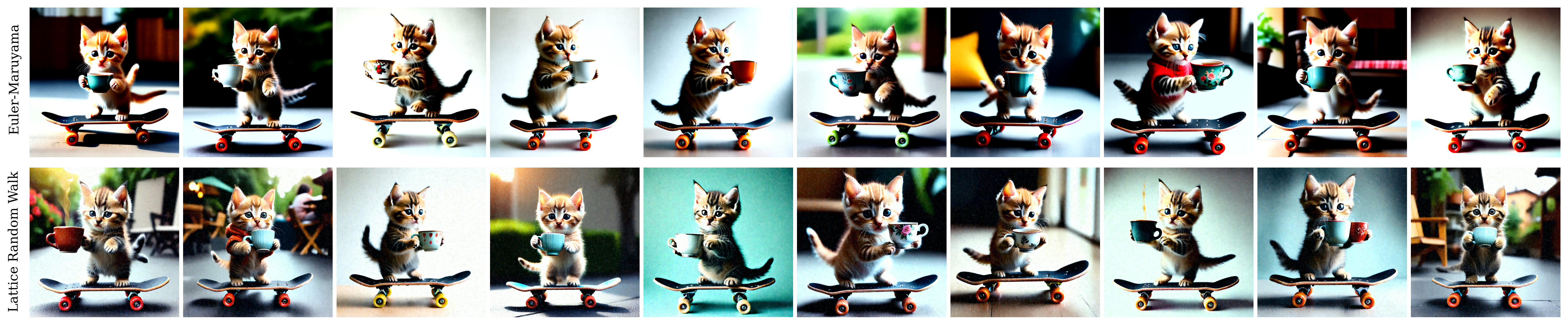}\label{fig:sd35:n50}}
  \vspace{0.5em}
  \subfloat[$\delta_t=1/25$.]{\includegraphics[width=\linewidth]{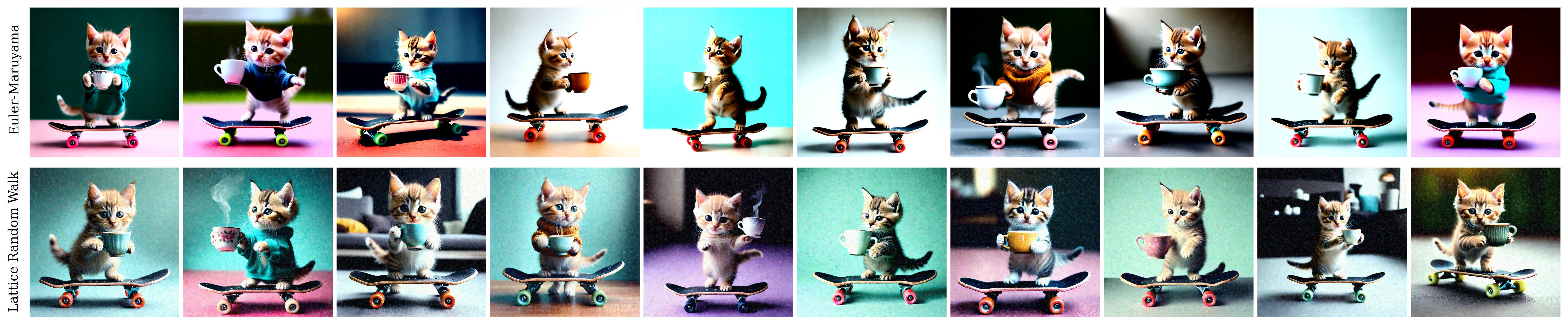}\label{fig:sd35:n25}}
  \caption{\textbf{Stable Diffusion 3.5 images generating with Euler-Maruyama and LRW}. Using $\delta_x(t) = \sqrt{\delta_t}\sigma(t)$,
  which in this case is time-varying. The images were generated with the prompt ``\textit{A kitten riding a skateboard holding a cup of tea}''
  using (a) 50 and (b) 25 discretisation steps (until terminal time $T=1$).}
  \label{fig:sd35}
\end{figure}

In Figure~\ref{fig:sd35} we compare the Euler-Maruyama and LRW discretisations for generating
prompted images from the pretrained Stable Diffusion 3.5 model~\cite{esser2024scaling}, which provide
$s(x,t)$ and $\varsigma(t)$ in~\eqref{eq:diffusion_sde}.
In all cases, we set the level of noise $\alpha(t) = - a \dot{\varsigma}(t)\varsigma(t)$ 
following~\cite{karras2022elucidating, song2020score} with $a = 0.3$. In this case, unlike the previous experiments,
the diffusion term $\sigma(t) = \sqrt{2\alpha(t)}$ is time-varying, but we can still apply
the rule of thumb from Section~\ref{subsec:dx_constraints} with a time-varying
$\delta_x(t) = \sqrt{\delta_t}\sigma(t)$.

A quantitative numerical analysis comparing Fréchet inception distance between EM and LRW over varying steps is found in Appendix~\ref{sec:fid}.

We can see from Figure~\ref{fig:sd35} that for the 50 timesteps, the LRW discretisation
is able to generate images that are equivalent in quality to that of Euler-Maruyama.
For 25 timesteps, the LRW images are still of high quality, but exhibit more graininess
or similary to the original noise source $t=0$. We posit that this is due to the fact that
the larger second moment of the Euler-Maruyama discretisation~\eqref{eq:second_moment}
is able to propagate the images further away from the initial noise source $t=0$ in
absolute terms, this pairs with the prescense of the stocasticity $\alpha(t) > 0$ acting
as an error correcting mechanism~\cite{schaeffer2025effect}.

Overall, Figure~\ref{fig:sd35} confirms that the LRW discretisation is scalable to
large-scale, modern stochastic differential equation models for machine learning pipelines.

\section{Discussion}\label{sec:discussion}

In this work, we introduced a novel LRW discretisation
for the simulation of stochastic differential equations and proved its weak convergence
to the true SDE. Unlike existing discretisations,
with the LRW the computation of the drift and diffusion functions only
appears in the sampling of a ternary random variable (or binary under certain parameter
setting, see Section~\ref{subsec:dx_constraints}). This means that the LRW
not only entirely avoids the need for Gaussian random variable generation but also
is significantly more robust to quantisation error and exploding drift functions,
as confirmed numerically in Sections~\ref{subsec:exp_ou} and \ref{subsec:exp_non_lipschitz}.
An additional consequence of the bottleneck computation being suppressed into the sampling
of a ternary random variable is that the LRW discretisation unlocks the
potential to use stochastic computing architectures that have historically struggled with
the problem of aggregating output bits to regain high precision output, a problem that
is sidestepped entirely by the LRW discretisation.

We conclude with discussing the limitations of the introduced LRW discretisation
as well as prominent directions for future work.

\subsection{Limitations}

As discussed in Section~\ref{subsec:diag_diffusion} the presented LRW discretisation
is limited to diagonal diffusion matrices. This is a strong assumption, but in practice
is rarely a problem as the vast majority of SDEs in application are formulated with
diagonal or scalar diffusion coefficients. Yet, the extension to dense diffusion
is a natural direction for future work which we perceive as achievable perhaps
inspired by the Lamperti transform (Appendix~\ref{sec:lamperti}), however one would
have to be careful to avoid matrix inversions at each iteration which can be costly.

Section~\ref{subsec:weak_convergence} presents a proof of the weak convergence of the 
LRW discretisation and at weak order~1. Under some restrictions of the SDE~\eqref{eq:sde}
weak order 2 convergence can be achieved through relatively straightforward modifications
to the Euler-Maruyama discretisation, such as the BAOAB-limit scheme~\cite{leimkuhler2013rational}
for the case of constant, scalar diffusion coefficient. It may be possible to adopt similar
ideas to construct LRW discretisations that are weak order 2 under similar SDE settings,
although the LRW requirement that all steps need to be stochastic will need to be carefully
considered.

We have not discussed the analysis of strong convergence, which measures the accuracy of individual sample paths 
rather than expectations. Strong convergence analysis for LRW is significantly more challenging than 
for traditional methods like Euler-Maruyama due to the discrete nature of the increments and the
coupling between the spatial stepsize $\delta_x$ and temporal stepsize $\delta_t$. However, 
weak convergence is generally more relevant for most applications, particularly in machine learning 
and statistical inference where we are primarily interested in statistical properties (expectations, 
variances, etc.) rather than the accuracy of individual trajectories. This is especially true for 
applications like diffusion models and Bayesian sampling where the goal is to generate samples from 
a target distribution rather than to accurately track specific paths.

In the numerical study in Section~\ref{sec:experiments}, we have focussed the experiments
on comparison with Euler-Maruyama, despite there being other discretisations to choose from
(see Table~\ref{tab:sde-methods}). We have done so because the LRW discretisation
is general-purpose in the sense that the only restriction on~\eqref{eq:sde} is that the
diffusion matrix $\sigma(x,t)$ is diagonal, in particular it supports time and/or space-varying
diffusion coefficients. This in contrast to other discretisations which either assume
a time-homogeneity (e.g. BAOAB-limit scheme~\cite{leimkuhler2013rational}),
or require multiple drift evaluations per step (e.g. Heun's method~\cite{kidger2020neural}).
Along with the fact that Euler-Maruyama remains the most popular discretisation for SDEs
in modern machine learning diffusion models~\cite{karras2022elucidating, song2020score}.
Therefore we have focussed on Euler-Maruyama as a general-purpose baseline for comparison,
in follow-up work as we focus on more specific applications, we will consider other
discretisations that are more appropriate for the application at hand.

We have also not included any experiments on stochastic computing architectures,
although we believe that the discrete nature of the LRW discretisation makes it
ideally suited for stochastic computing architectures. However, due to the additional
complexity of accurately simulating stochastic computing primitives we have decided to limit
the scope of this work to the simulation of SDEs and justification of (an exact implementation of)
the LRW discretisation in its own right. Future work will detail explicitly potential implementations
and simulations of stochastic computing architectures.

\subsection{Future Work}

Naturally, the above limitations open up several directions for future work
regarding dense diffusion matrices, higher-order schemes and strong convergence as discussed
above.

Additionally, the development of specialized stochastic computing hardware implementations represents 
perhaps the most transformative direction and one we are actively pursuing~\cite{melanson2025thermodynamic,belateche2025scaling,coles2023thermodynamic_published,aifer2025solving,duffield2023thermodynamic,donatella2024thermodynamic}.
The discrete nature of LRW outputs makes it 
ideally suited for stochastic computing architectures. Indeed, developing dedicated hardware 
that can efficiently generate and manipulate the binary/ternary random variables  could 
unlock many orders of magnitude speedup for SDE simulation in applications such as
large-scale diffusion models, as unlocked
by the introduction of the lattice random walk discretisation.

\bibliographystyle{apalike}
\bibliography{ThermoMasterBib.bib}

\section*{Acknowledgments}

We thank Lars Holdijk, Rob Brekelmans, Gavin Crooks, Jan Ole Ernst, and Max Welling for valuable feedback during the writing of the manuscript. We thank the Advanced Research and Invention Agency’s (ARIA) Scaling Compute programme for funding this work.

\section*{Author contributions}
S.D. invented the binary form of the algorithm. M.A. derived the binary discretization for unit step size. M.A. assisted with the derivation to the more general ternary form. S.D. and D.M. ran the experiments. M.A., Z.B. and P.C. contributed to fundamental early work on a version for linear systems and its connection to stochastic computing. All authors reviewed the manuscript.

\section*{Competing Interests}
All authors declare no financial or non-financial competing interests.

\appendix

\section{Weak Convergence}\label{sec:weakproof}

We now prove the lattice random walk discretisation for the SDE \eqref{eq:sde} converges with weak order 1. We restate the theorem from Section~\ref{subsec:weak_convergence} for clarity.

\begin{theorem_reset}[Weak convergence of the LRW discretisation]
Consider the SDE \eqref{eq:sde} with drift function $f(x,t)$ and diagonal diffusion
matrix $\sigma(x,t)$ that are sufficiently smooth.
Let $\varphi: \mathbb{R}^d \to \mathbb{R}$ be a test function with bounded derivatives. 
Then the LRW discretisation~\eqref{eq:lrw} with spatial stepsize
$\delta_{x, i} = \Theta(\sqrt{\delta_t})$ has weak order 1, i.e.,
\[
\bigl|\E[\varphi(x_N)] - \E[\varphi(X(T))]\bigr| = O(\delta_t),
\]
where $x_N$ is the discretized solution at time $T = N\delta_t$ and $X(T)$ is the true SDE solution.
\end{theorem_reset}
\begin{proof}We prove weak convergence with order 1 by using the infinitesimal generator of the true SDE to bound the local and then global error between expectations with respect to the true SDE and samples from the LRW discretisation.

% \paragraph{Weak error definition.}
% For a sufficiently smooth test function $\varphi$, the weak error after $N$ steps ($T=N\delta_t$) is
% \[
%   \bigl|\E[\varphi(x_N)] - \E[\varphi(X(T))]\bigr|
%   = O(\delta_t^p),
% \]
% and the method is said to have \emph{weak order} $p$ if the leading term scales like $O(\delta_t^p)$ as above.

\paragraph{Generator.}
Let $\varphi\colon\mathbb{R}^d\to\mathbb{R}$ be a sufficiently smooth test function and denote by $X(t)$ the true solution of \eqref{eq:sde} with diagonal $\sigma(x, t)$. Then the generator of the true solution has the form [\citealp{pavliotis2014stochastic}, page 50]
\[
  L\varphi(x) \;=\; f(x,t)\cdot\nabla\varphi(x)\;+\;\tfrac12\sum_{i=1}^d\sigma_i(x,t)^2\,\partial_{ii}\varphi(x)\,,
\]
(note the above is the simplified generator in the case of diagonal $\sigma(x, t)$).

\paragraph{Increment moments.}
Writing $x_{+}=x+\Delta$ with diagonal $\sigma(x, t)$, then from the definitions of $p_\pm(x,t)$ in \eqref{eq:ternary-p} we get
\begin{align*}
  \E[\Delta]
  &= (p_{+}-p_{-})\,\delta_x
   = \delta_t f(x,t)\\
  \E[\Delta^2]
  &= (p_{+}+p_{-})\,\delta_{x}^2
   = \delta_t \sigma(x, t)^2,\\
   \E[\Delta_i \Delta_j] &= \E[\Delta_i] \E[\Delta_j] = \delta_t^2 f_i(x,t)f_j(x,t) \qquad \text{for $i\neq j$,}
\end{align*}
and also third moments
\begin{align*}
   \E[\Delta^3]
   &= \delta_t \delta_x^{2} f(x,t),\\
   % \E[\Delta^2_i \Delta_j] &= \E[\Delta^2_i] \E[\Delta_j] = (\delta_t \sigma_i(x, t)^2 ,\\
   \E[\Delta^2_i \Delta_j] &= \E[\Delta^2_i] \E[\Delta_j] = \delta_t^2 \sigma_i(x, t)^2 f_j(x,t),  \qquad \text{for $i\neq j$,} \\
   \E[\Delta_i \Delta_j \Delta_k] &= \E[\Delta_i] \E[\Delta_j] \E[\Delta_k] = \delta_t^3 f_i(x,t)f_j(x,t)f_k(x,t), \qquad \text{for $i\neq j \neq k \neq i$.}
\end{align*}
Where in all cases the moments are conditional on $x$.

\paragraph{One‐step expansion.}
Expanding in $\Delta$ gives
\begin{align*}
  \E[\varphi(x_{+})]
  &= \varphi(x)
  + \nabla\varphi(x)^\top \E[\Delta]
  + \frac12 \sum_{i,j=1}^d  \E[\Delta_i \Delta_j]\partial_{ij}\varphi(x)
  % + \tfrac12\Tr\Bigl[D^2\varphi(x)\,\E[\Delta\,\Delta^\top]\Bigr]
  + O\bigl(\E[\|\Delta\|^3]\bigr).
\end{align*}

We conclude
\begin{align*}
  \E[\varphi(x_{+})]
  =& \varphi(x)
  + \delta_t\,L\varphi(x)
  + \frac12 \, \delta_t^2 \sum_{i\neq j} f_i(x,t) f_j(x,t) \,\partial_{ij}\varphi(x)
    +O(\delta_t\,\delta_x^2) + O(\delta_t^2),
\end{align*}

where we used \(\E[\|\Delta\|^3]=O(\delta_t\,\delta_x^2) + O(\delta_t^2)\).

\paragraph{Local weak error.}
From the one‐step expansion we have
\[
  \E[\varphi(x_{+})]
  = \varphi(x)
  + \delta_t\,L\varphi(x)
  + O(\delta_t^2)
  + O(\delta_t\,\delta_x^2),
\]
so the local weak error is
\[
  R(x)
  := \E[\varphi(x_{+})] - \varphi(x) - \delta_t\,L\varphi(x)
  = O(\delta_t^2)
  + O(\delta_t\,\delta_x^2)
\]

\paragraph{$\delta_x$ constraint.} We now use $\delta_{x,i} = \Theta(\delta_t^{1/2})$ so that \(O(\delta_t\,\delta_x^2)=O(\delta_t^2)\).
% The requirement \(p_-+p_+\le1\) in \eqref{eq:ternary-p} yields, for each coordinate \(i\),
% \[
%   p_- + p_+
%   = \delta_t\,\delta_x^{-2}\,\sigma_i^2
%   = O(1)
%   \quad\Longrightarrow\quad
%   \delta_x = \Theta(\delta_t^{1/2}),
% \]
% so that \(O(\delta_t\,\delta_x^2)=O(\delta_t^2)\).

\paragraph{Global weak error.}
Hence $R(x) = O(\delta_t^2)$,
and over \(N=T/\delta_t\) steps the global weak error is
\[
  \bigl|\E[\varphi(x_N)] - \E[\varphi(X(T))]\bigr|
  = N\,O(\delta_t^2)
  = O(\delta_t),
\]
i.e.\ the scheme is weak‐order~1 for $\delta_{x,i} = \Theta(\delta_t^{1/2})$.

\end{proof}

\section{Lamperti Transform}\label{sec:lamperti}
Given an SDE with time-varying but not state-varying diffusion coefficient
\begin{align*}
dx_t &= f(X_t,t)\,dt + \sigma(t)\,dW_t, 
\end{align*}
define the rescaling
\begin{align*}
z_t &= \kappa\,\sigma(t)^{-1}\,x_t.
\end{align*}

Then by Itô’s formula one finds
\begin{align}\label{eq:lamperti}
dz_t
&=\left[
\;\frac{d}{dt}\bigl(\sigma(t)^{-1}\bigr)\,\sigma(t)\,z_t
\;+\;
\kappa\,\sigma(t)^{-1}\,f\!\bigl(\tfrac{1}{\kappa} \sigma(t) z_t,t\bigr)
\right]\,dt
\;+\;\kappa\,dW_t,
\end{align}
where $\sigma(t)^{-1}$ is the matrix inverse of $\sigma(t)$ and $\kappa \in (0, \infty)$ is a scalar.

And we can transform back with
\begin{align*}
    x_t = \frac{1}{\kappa} \sigma(t) z_t.
\end{align*}

For more details, including the case of state varying diffusion coefficient $\sigma(x, t)$ see \cite{moller2010state}.

\section{Reparameterisation to Integers}\label{subsec:reparam}

Formally the discretisation in~\eqref{eq:lrw} is defined on a lattice of points but not
integers since the spatial stepsize $\delta_x$ is real valued. However we can consider
the reparameterization
\begin{align}\label{eq:reparam}
    x = \delta_x z \iff z = \delta_x^{-1} x.
\end{align}
This gives an integer update rule in $z$ (assuming $z_0$ is an integer)
\begin{align}
    z_{t+\delta_t} &= 
    z_t + \Delta(z_t, t), \nonumber \\
    \mathbb{P}[\Delta_i(z, t) = \Delta_i]
    &=
    \begin{cases}
        p_{-,i}(z,t), &\text{if } \Delta_i = -1, \\
        1 - p_{-,i}(z,t) - p_{+,i}(z,t), &\text{if } \Delta_i = 0, \\
        p_{+,i}(z,t), &\text{if } \Delta_i = 1, \\
    \end{cases} \label{eq:ternary-z}
\end{align}
with probabilities
\begin{align*}
    p_\pm(z, t) &= \frac{1}{2} \delta_t\delta_x^{-1} \left[\pm f(\delta_x z, t) +\delta_x^{-1}\sigma(\delta_x z, t)^2\right].
\end{align*}

\section{From Flows to SDEs}\label{sec:from_flows_to_sdes}

In this section we summarize the framework introduced in \cite{karras2022elucidating} unifying diffusion and flow models at inference time as a family of SDEs with flexible amount of noise (including zero noise i.e. an ODE as a special case).

Flow matching models \cite{karras2022elucidating} typically learn to reverse a noising process, where the noising
process is defined as the mollified Gaussian for some noise schedule $\sigma(\tau)$
\begin{align*}
p_\tau(x_\tau \mid x_0) = \mathcal{N}(x_\tau \mid x_0, \sigma(\tau)^2 \mathbb{I}),
\end{align*}
where we have used $\tau$ rather than $t$ to indicate this is the forward or noising process,
(before later using $t=1-\tau$ to indicate the more complicated reverse or denoising process to match~\eqref{eq:sde}).

The output of the flow matching training is a velocity $u_\theta(x_\tau, \tau)$ that can be
used within an ordinary differential equation (ODE) solver that reverses the process for $\tau=1 \to 0$
\begin{align*}
x_1 \sim \mathcal{N}(x_1 \mid 0, \sigma_\text{max}^2 \mathbb{I}), \qquad \qquad dx = u_\theta(x, \tau)d\tau.
\end{align*}
It can be unnatural to think of an ODE running in reverse time with $d\tau < 0$,
instead we can rewrite the above ODE as
\begin{align*}
x_0 \sim \mathcal{N}(x_0 \mid 0, \sigma_\text{max}^2 \mathbb{I}), \qquad \qquad dx = -u_\theta(x, 1-t)dt,
\end{align*}
where now we are advancing $t=0 \to 1$ with $dt > 0$.

As described in~\cite{karras2022elucidating}, the trained flow matching model actually
recovers a scaled version of the score function $s(x, t) = \nabla_x \log p_t(x, t)$.
That is, the ordinary differential equation, moving forward in time $t=0\to1$, can be written as
\begin{align*}
dx = -\dot{\varsigma}(t) \varsigma(t)\nabla_x \log p_t(x_t, t) dt,
\end{align*}
where $\varsigma(t) = \sigma(1-t)$. Since $\sigma(\tau)$ is typically increasing with $\tau$ then $\varsigma(t)$ is decreasing with $t$ and $\dot{\varsigma}(t)$ is negative. 

We can convert the velocity to the score with knowledge of the noise schedule $\varsigma(t)$
\begin{align*}
\nabla_x \log p_t(x_t, t) = \frac{1}{\dot{\varsigma}(t)\varsigma(t)}u_\theta(x_t, 1-t).
\end{align*}

Further, this ODE can be generalised to an SDE by adding any amount of Langevin noise
that is invariant for the instantaneous distribution $p(x,t)$~\cite{ma2015complete}
\begin{align*}
dx_t = -\dot{\varsigma}(t) \varsigma(t)\nabla_x \log p_t(x_t, t) dt 
+ \alpha(t)\nabla_x \log p_t(x_t, t) dt + \sqrt{2 \alpha(t)}dw_t,
\end{align*}
matching that of \eqref{eq:diffusion_sde} where $\alpha(t) \geq 0$ controls the level of added noise but
remarkably does not affect the marginal distributions $p(x,t)$~\cite{karras2022elucidating, duffield2026complete}. A common choice is $\alpha(t) = - a \dot{\varsigma}(t) \varsigma(t)$ for some value of $a > 0$ (noting that $\dot{\varsigma}(t) <0$ implies $\alpha(t)>0$).

\section{Fréchet Inception Distance Experiment}\label{sec:fid}

In Figure~\ref{fig:fid}, we extend the experiment in Section~\ref{subsec:exp_diffusion} to give a numerical quantification of the quality of images generated by LRW - we use the commonly used \emph{Fréchet Inception distance} to numerically quantify images generated with gold standard images (in this case running SD3.5 with a small stepsize over 100 steps). In each case we generate 100 images using the same prompt as in Section~\ref{subsec:exp_diffusion}. We see that Euler-Maruyama out-performs at low number of steps (and therefore large stepsize), with LRW catching up for more steps and converging at a similar rate. This behaviour is not unsurprising given the harsh quantisation of the LRW increment and indeed it is somewhat remarkable that LRW performance is close - noting that after 50 steps EM and LRW are not visually distinguishable as seen in Figure~\ref{fig:sd35:n50}.

\begin{figure}[h]
  \centering
  \includegraphics[width=0.5\linewidth]{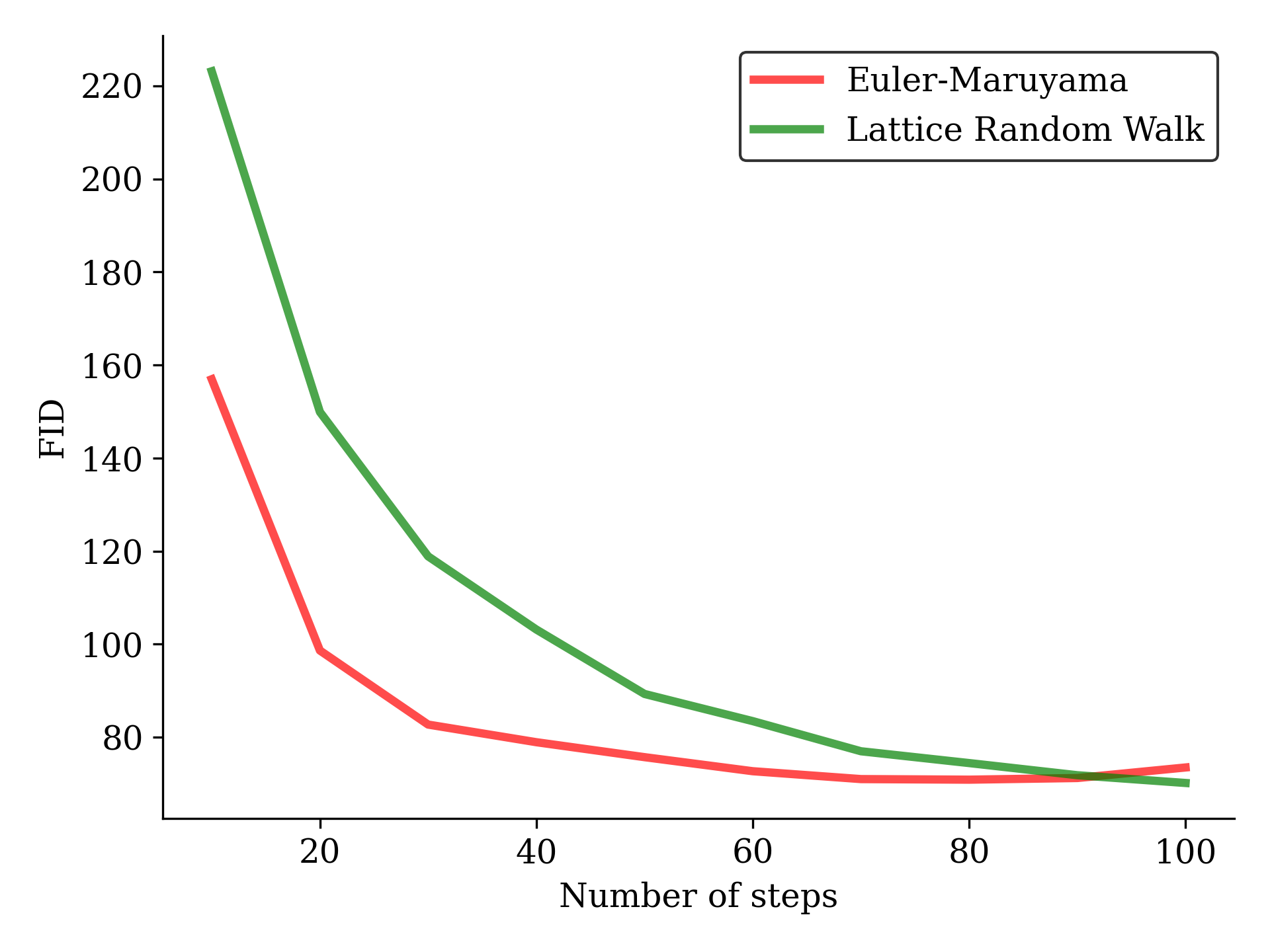}
  \caption{\textbf{FID distance from gold-standard images} All images were generated with Stable Diffusion 3.5 with the prompt ``\textit{A kitten riding a skateboard holding a cup of tea}'' as in Figure~\ref{fig:sd35}.}
  \label{fig:fid}
\end{figure}

\section{Simlulating Ornstein-Uhlenbeck Processes with Stochastic Computing}\label{sec:stochastic_ou}
As a concrete example of the LRW discretisation with stochastic computing, we can consider simulating an Ornstein-Uhlenbeck process
$$
dx = -(Ax - b)dt + \sigma dw.
$$
An iteration of binary LRW amounts to sampling a bipolar vector $v_t \in \{-1, 1\}^d$ with moments
\begin{align*}
    \mathbb{E}[v_t] = - \sigma^{-1}\sqrt{\delta_t} (Ax_t - b), \quad
    \mathbb{E}[v_t^2] = 1,
\end{align*}
then we get LRW update $x_{t+\delta_t} = x_t + \sqrt{\delta_t}\sigma v_t$.

The bipolar increment expectation can be written as a single matrix vector product
\begin{align}
    \mathbb{E}[v_t] = - \sigma^{-1}\sqrt{\delta_t} By_t, \quad
    B = (A \,,\; {-}b), \quad y_t = (x_t^\top \,, \; 1)^\top. \label{eq:mat_vec}
\end{align}
A common method for (scaled) addition in stochastic computing is the so-called multiplexer approach (e.g.\ Section 3 in~\cite{gross2019stochastic}). Generalizing to a dot product we can write the protocol for input stochastic bipolar vector $Y \in \{-1, 1\}^n$ with expected value $\mathbb{E}[Y_i] = y_i$ and real weights $w \in \mathbb{R}^n$ as
\begin{align*}
    k &\sim \text{Categorical}\left(\frac{|w_1|}{\|w\|_1}, \dots, \frac{|w_n|}{\|w\|_1}\right),
    \\
    b &= \text{sign}(w_k) Y_k,
\end{align*}
where $\|w\|_1 = \sum_j |w_j|$. This gives expected value as the (scaled) dot product
\begin{align*}
    \mathbb{E}[b] = \frac{w^\top y}{\|w\|_1},
\end{align*}
which can be repeated in parallel for the $d$ rows of the matrix-vector product in \eqref{eq:mat_vec}.

To apply this protocol, we encode $y_t$ as a stochastic bipolar vector: assuming $\|y_t\|_\infty \leq M$, each component $y_{t,j}$ is represented by a random bit $Y_j \in \{-1, +1\}$ with $\mathbb{P}[Y_j = +1] = (1 + y_{t,j}/M)/2$, giving $\mathbb{E}[Y_j] = y_{t,j}/M$. Sampling from the categorical distribution na\"ively requires $O(n)$ time, but using the alias method~\cite{walker1977efficient} with $O(n)$ preprocessing, each sample can be drawn in $O(1)$ time. This gives $O(d^2)$ preprocessing (for $d$ alias tables, one per row of $B$) and $O(d)$ work per LRW iteration, or $O(1)$ parallel time with $O(d)$ threads since the $d$ categorical samples are independent.

For the protocol to be valid, we require $|\mathbb{E}[v_{t,i}]| \leq 1$ for each coordinate $i$. Combined with a bound $\|y_t\|_\infty \leq M$ (which holds if the OU process remains bounded), we obtain a maximal stepsize
\begin{align*}
    \delta_t &\leq \left(\frac{\sigma}{M\bar{B}}\right)^2,
\end{align*}
where $\bar{B} = \max_i \sum_j |B_{ij}|$ is the maximum absolute row sum of $B$.

\end{document}